\documentclass[a4paper]{article}
\usepackage{authblk}

\usepackage{amsmath,amssymb,amsthm}
\usepackage{a4wide}

\newcommand{\G}{\mathcal{G}}
\newcommand{\C}{\mathcal{C}}
\newcommand{\Z}{\mathcal{Z}}
\newcommand{\B}{\mathcal{B}}
\newcommand{\A}{\mathcal{A}}
\newcommand{\T}{\mathcal{T}}
\newcommand{\F}{\mathcal{F}}
\newcommand{\cS}{\mathcal{S}}
\newcommand{\M}{\mathcal{M}}

\newcommand{\TODO}[1]%
{\ifdetails\par\fbox{\begin{minipage}{0.9\linewidth}\textbf{TODO:}
      #1\end{minipage}}\par\fi}

\newtheorem{lemma}{Lemma}
\newtheorem{prop}[lemma]{Proposition}
\newtheorem{theorem}[lemma]{Theorem}
\newtheorem{corollary}[lemma]{Corollary}
\theoremstyle{remark}

\newtheorem{definition}{Definition}
\newtheorem{conjecture}{{Conjecture}}[section]

\title{Subcritical graph classes containing all planar graphs}

\begin{document}

\author[1]{Stephan Wagner\thanks{Supported by the National Research Foundation of South Africa, grant number 96236.}}

\author[2]{Agelos Georgakopoulos\thanks{Supported by the European Research Council (ERC) under the European Union’s Horizon 2020 research and innovation programme (grant agreement No 639046).}}
 
\affil[1]{Department of Mathematical Sciences\\ Stellenbosch University\\ Private Bag X1, Matieland 7602, South Africa\\ \texttt{swagner@sun.ac.za}}
\affil[2]{{Mathematics Institute}\\
 {University of Warwick}\\
  {CV4 7AL, UK}
}

\maketitle

\begin{abstract}
We construct minor-closed addable families of graphs that are subcritical and contain all planar graphs. This contradicts (one direction of) a well-known conjecture of Noy.

\bigskip


\end{abstract}

\section{Introduction} 

Subcritical classes of (finite, simple) graphs are defined by a technical condition involving their generating functions; we recall the formal definition in the next section. 

Loosely speaking, graphs from  subcritical families can be thought of as ``tree-like''. Indeed, it is shown in \cite{Georgakopoulos2015limits,StuRan} that their Benjamini--Schramm limits are similar to that of random trees. Subcritical graph classes have also been observed to exhibit tree-like behaviour of a different kind:  Panagiotou, Stufler and Weller \cite{Panagiotou2016scaling} showed that the scaling limit of random graphs from subcritical classes, in the Gromov--Hausdorff sense, is Aldous' continuum random tree (see also \cite{StuRan}).

Other properties of subcritical families that have been studied include the degree distribution~\cite{BePaStDeg}, extremal parameters such as the maximum degree and the diameter~\cite{drmota13}, and subgraph counts~\cite{DrmRamRue}. 
Important examples of subcritical graph classes include cacti, outerplanar graphs and series-parallel graphs.

Noy  \cite{Noy2014random} made the following well-known conjecture, which attempts a characterization of the subcritical families that are addable and minor-closed.

\begin{conjecture}[\cite{Noy2014random}]
An addable, minor-closed class of (labelled) graphs is subcritical if and only if it has a planar forbidden minor.
\end{conjecture}

A class of graphs $\mathcal{G}$ is called {\em minor-closed} if every minor of a graph $G \in \mathcal{G}$ is again in $\mathcal{G}$. The class $\mathcal{G}$ is called {\em addable}, if it satisfies the following two requirements: 1) the disjoint union of any two graphs $G,H \in \mathcal{G}$ belongs to $\mathcal{G}$, and 2) for every $G \in \mathcal{G}$, the graph obtained by adding an edge between two vertices in distinct  components of $G$ is again an element of $\mathcal{G}$.

Recall that a minor-closed graph class has a planar forbidden minor if and only if it has bounded tree-width \cite{GMV}. Thus, in the light of the above discussion, Noy's conjecture can be interpreted as stating that tree-likeness in the graph minor sense coincides with tree-likeness in the sense of enumerative and probabilistic combinatorics.

The aim of this paper is to provide a counterexample to one direction of Noy's conjecture: we construct addable, minor-closed,  subcritical classes of graphs that contain all planar graphs\footnote{Even more, in Section~\ref{secdisc} we observe that our construction can be generalised to contain all graphs from any fixed proper minor-closed graph family.}. These classes are defined as follows:

\begin{definition}\label{ref:Gkdefi}
We let $\G_k$ be the class of all graphs with the property that every {\em block} (i.e.\ maximal connected subgraph without a cutvertex) is either planar or can be reduced to a forest by removing at most $k$ vertices.
\end{definition}

Verifying that $\G_k$ is indeed addable and minor-closed is not difficult; the main result will be the fact that it is subcritical for $k \geq 4$. Heuristically, the reason is that the ``tree-like'' blocks become more numerous and thus asymptotically more important than the planar blocks, and that $\G_k$ inherits the tree-likeness from them.

The number of graphs with $n$ vertices that can be reduced to a forest by removing at most $k$ vertices, from now on called {\em $k$-apex forests}, has been obtained asymptotically by Kurauskas and McDiarmid \cite{KurMcDRan}. Combining their result with our asymptotic enumeration of the 2-connected ones, we deduce, in Section~\ref{sec 2 conn}, that the probability for a uniformly random $k$-apex forest to be 2-connected decays exponentially in $n$, and we determine the rate of exponential decay for each $k$.

Let us remark that we will only consider labelled graphs in this manuscript. However, the same approach also applies to the unlabelled setting. We also emphasize that this only provides a counterexample to one direction of Noy's conjecture. The other direction (an addable, minor-closed graph class with a forbidden planar minor is subcritical) remains plausible. 

\section{Subcriticality}
A maximal connected subgraph without a cutvertex is called a \emph{block}.
A class $\G$ of graphs is called \emph{block-stable} if it satisfies the following property: a graph $G$ lies in $\G$ if and only if all blocks of $G$ lie in $\G$. Let $\C$ denote the class of all connected graphs in $\G$ and $\B$ the class of all blocks in $\G$. Moreover, let $B(z)$, $C(z)$ and $G(z)$ be the  exponential generating functions associated with $\B$, $\C$ and $\G$ respectively. The three are connected by the functional equations
\begin{equation}\label{eq:subcritical_funeq}
G(z) = \exp(C(z)),\qquad C^{\bullet}(z) = z \exp(B'(C^{\bullet}(z)),
\end{equation}
where $C^{\bullet}(z) = z C'(z)$ is the exponential generating function for \emph{rooted} connected graphs in $\G$. Let $\eta$ and $\rho$ be the radii of convergence of $B(z)$ and $C(z)$ respectively. If $C^{\bullet}(\rho) < \eta$, the class $\G$ is called \emph{subcritical}. This technical condition ensures that $C^{\bullet}$ has a dominant singularity at $\rho$ of square-root type (see \cite{DFKKR}), with many important consequences. As mentioned earlier, subcritical graphs are tree-like in many ways. The main result of this manuscript is:

\begin{theorem}\label{thm:main}
For every $k \geq 4$, the graph class $\G_k$ is subcritical.
\end{theorem}

The main idea is as follows: we first determine the radius of convergence $\eta_k$ of the exponential generating function $B_k(z)$ associated with the blocks of $\G_k$. In doing so, we show that planar blocks form a negligible part of the set of possible blocks. Finally, we prove that $B''(z)$ goes to infinity as $z \to \eta_k^-$. This in turn is used to prove subcriticality.

\section{Proofs}

To prove that the classes of Definition~\ref{ref:Gkdefi} contradict Noy's conjecture, we have to show that they are addable, minor-closed and ---for $k\geq 4$--- subcritical. We start with the first two properties, which are easier to prove.

\begin{prop}\label{prop_add_mc}
For every positive integer $k$, the graph class $\G_k$ is both addable and minor-closed.
\end{prop}
 
\begin{proof}
To prove that $\G_k$ is addable, note first that disjoint unions of graphs in $\G_k$ are trivially elements of $\G_k$ again. Next consider any graph $G \in \G_k$, and let $G'$ be obtained from $G$ by connecting two vertices in distinct connected components of $G$ by an edge. Note that the blocks of $G'$ are the blocks of $G$ and the newly added edge. Since a single edge is a planar graph and thus an allowed block, the graph $G'$ still lies in $\G_k$. Hence the class is addable.

To show that $\G_k$ is also minor-closed, we need to prove that it is closed under the operations of removing a vertex, removing an edge, or contracting an edge. When an edge or a vertex is removed, all blocks of the new graph are subgraphs of blocks of the old graph. Every subgraph of a planar graph is again planar. Likewise, if a graph can be reduced to a forest by removing at most $k$ vertices, so can any subgraph (by removing the same vertices or -- if some of them are not part of the subgraph -- a subset thereof). Thus the condition in Definition~\ref{ref:Gkdefi} remains valid if vertices or edges are removed.

The last operation to consider is contraction of edges; it suffices to consider the block that contains the contracted edge. If this block is planar, it remains so after the edge contraction. If the block can be reduced to a forest by removing at most $k$ vertices, then this is still true after the edge contraction: removing the same vertices (possibly one less, because two of them have been reduced to a single vertex) yields the original forest, a subgraph thereof (again a forest), or the original forest with a contracted edge (which is also a forest). In each of these cases, all newly created blocks (which are subgraphs of the old block with the contracted edge) satisfy the condition of Definition~\ref{ref:Gkdefi}, completing our proof.
\end{proof}

We now proceed with the proof of the main result of this paper, Theorem~\ref{thm:main}.
Recall that there are $n^{n-2}$ labelled trees with $n$ vertices and consequently $n^{n-1}$ rooted labelled trees. The exponential generating function associated with rooted labelled trees is sometimes called the \emph{tree function}:
$$T(z) = \sum_{n=1}^{\infty} \frac{n^{n-1}}{n!} z^n.$$
The tree function is closely related to the \emph{Lambert $W$-function}, which is defined using the functional equation $W(z) e^{W(z)} = z$. This equation defines a multivalued function, and if we let $W$ denote its principal branch, we can express $W$ as the power series $W(z) = \sum_{n=1}^{\infty} \frac{(-n)^{n-1}}{n!} z^n$, see \cite{CGHJK,JosDer}. Thus we have $T(z) = -W(-z)$. Hence $T(z)$ satisfies the functional equation $T(z) = z e^{T(z)}$.


The exponential generating function for unrooted (labelled) trees is given by
\begin{equation}\label{eq:unrooted_gf}
t(z) = \sum_{n=1}^{\infty} \frac{n^{n-2}}{n!} z^n = \int_0^z \frac{T(u)}{u}\,du = T(z) - \frac{T(z)^2}{2}.
\end{equation}
Note here that $T(z)$ represents rooted trees, while $\frac{T(z)^2}{2}$ is the exponential generating function for edge-rooted trees (equivalent to unordered pairs of rooted trees). Since the number of vertices of a tree is always the number of edges plus $1$, the difference yields exactly the exponential generating function for unrooted trees. A forest is a collection of trees, hence the exponential generating function associated with the class of all forests is $f(z) = \exp(t(z))$. We will denote the class of all labelled unrooted trees by $\T$ and the class of all labelled forests by $\F$.

In the following, we need bivariate versions of $T,t,f$ that also involve the number of leaves: let $T(z,u)$, $t(z,u)$ and $f(z,u)$ be those three exponential generating functions, where the exponent of $u$ equals the number of leaves. The symbolic method described in Part A of~\cite{Flajolet2009analytic} can be used to obtain functional equations for these functions. First of all, we have (cf. \cite[Theorem 3.13]{Drmota2009random})
\begin{equation}\label{eq:bivariate_fun_eq}
T(z,u) = z \exp(T(z,u)) + (u-1)z,
\end{equation}
which follows from the fact that the number of leaves of a rooted tree equals the sum of the number of leaves over all its branches (the root only counts as a leaf in this context if it is the only vertex), unless the tree consists of the root only. The last term $(u-1)z$ takes this into account. 
The functional equation~\eqref{eq:bivariate_fun_eq} has the explicit solution
\begin{equation}\label{eq:Tzu_explicit}
T(z,u) = (u-1)z + T(z e^{(u-1)z}).
\end{equation}
Furthermore, we have
\begin{equation}\label{eq:tzu_explicit}
t(z,u) = T(z,u) + (u-1)z T(z,u) - \frac{T(z,u)^2}{2},
\end{equation}
the explanation being similar to~\eqref{eq:unrooted_gf}: the first term stands for rooted trees. The second term corrects for the fact that we did not let the root count as a leaf in $T(z,u)$ unless it was the only vertex. The last term represents edge-rooted trees. Finally, since forests are simply collections of trees, we have
\begin{equation}\label{eq:fzu_explicit}
f(z,u) = \exp(t(z,u)).
\end{equation}
In the following, we will make use of the fact that $T$ is analytic in the complex plane, except for a branch cut along the positive real axis, starting at $\frac{1}{e}$ \cite{CGHJK,JosDer}. Its asymptotic expansion at the branch point $\frac1{e}$ is given by (see \cite[(4.22)]{CGHJK})
\begin{equation}\label{eq:tree_expansion}
T(z) = 1 - \sqrt{2(1-e z)} + \frac{2}{3}(1-ez) - \frac{11\sqrt{2}}{36} (1-ez)^{3/2} + O \Big( \Big(1- e z\Big)^2 \Big),
\end{equation}
valid in an any fixed neighbourhood of the branch point $\frac1{e}$ with the real numbers greater than $\frac1{e}$ removed. This will allow us to apply the principles of singularity analysis \cite[Chapter VI]{Flajolet2009analytic} to some of the generating functions we encounter. It will be important later that the term involving $\sqrt{1-e z}$ vanishes in the generating function $t(z)$ as given in~\eqref{eq:unrooted_gf}: a simple calculation shows that
\begin{equation}\label{eq:unrooted_expansion}
t(z) = T(z) - \frac{T(z)^2}{2} = \frac12 - (1- e z) + \frac{2\sqrt{2}}{3} (1-ez)^{3/2} + O \Big( \Big(1- e z\Big)^2 \Big).
\end{equation}

Now let $B_k(z)$ be the exponential generating function associated with blocks in $\G_k$, and let $A_k(z)$ be the exponential generating function for the ``second type'' of blocks in $\G_k$, i.e., 2-connected graphs that can be reduced to a forest by removing at most $k$ vertices. This set of graphs is denoted by $\A_k$.
In our first lemma, we bound the number of elements of $\A_k$, which in turn gives an estimate for the radius of convergence.

\begin{lemma}\label{lem:upper}
Let $\eta_k = T(\frac{1}{2^k e})$ be the smallest positive solution to the equation $2^k \eta = e^{\eta-1}$. There exists a positive constant $K_1$ (depending on $k$) such that $\A_k$ contains at most $K_1 n^{-5/2} \eta_k^{-n}n!$ elements with $n$ vertices for all positive integers $n$.
\end{lemma}

\begin{proof}
Every element of $\A_k$ that is not just a single edge consists of a forest $F$ and $r$ additional vertices, $1 \leq r \leq k$; each leaf of the forest needs to be adjacent to at least one of the additional vertices, for otherwise the minimum degree would be $1$, making it impossible for the graph to be $2$-connected. Let $\ell(F)$ denote the number of leaves of $F$. Given $F$ and $r$, there are $2^{\binom{r}{2}}$ possibilities for the edges between the additional vertices, there are $2^r-1$ possible ways to connect a leaf of $F$ to the additional vertices (any possible set of edges except for the empty set) and $2^r$ possible ways to connect each other vertex of $F$. Therefore, an upper bound for the number of elements of $\A_k$ with $n$ vertices is given by
$$\sum_{r=1}^k \sum_{\substack{F \in \F \\ |F| = n-r}} \binom{n}{r} 2^{\binom{r}{2}} (2^r-1)^{\ell(F)}(2^r)^{|F|-\ell(F)} = \sum_{r=1}^k \sum_{\substack{F \in \F \\ |F| = n-r}} \binom{n}{r} 2^{\binom{r}{2}} (1 - 2^{-r})^{\ell(F)}(2^r)^{|F|} $$
(the binomial coefficient $\binom{n}{r}$ takes the possible ways to assign labels to the $r$ special vertices into account). Note that this is indeed just an upper bound rather than the exact number: not all graphs obtained in this way are $2$-connected, and there is also some double-counting, see the discussion in the following lemma. The exponential generating function associated with this estimate is
\begin{equation}\label{eq:upper_est}
U_k(x) = \sum_{r=1}^k 2^{\binom{r}{2}} \frac{x^r}{r!} \sum_{m \geq 0} \sum_{\substack{F \in \F \\ |F| = m}} \frac{(2^r x)^m}{m!}  (1-2^{-r})^{\ell(F)} = \sum_{r=1}^k 2^{\binom{r}{2}} \frac{x^r}{r!} f \big( 2^r x, 1-2^{-r}\big).
\end{equation}
Note that $f \big( 2^r x, 1-2^{-r}\big)$ has positive coefficients, so by Pringsheim's Theorem it must have a positive real singularity on its circle of convergence.
In view of~\eqref{eq:Tzu_explicit},~\eqref{eq:tzu_explicit} and~\eqref{eq:fzu_explicit}, $f(2^r x,1-2^{-r})$ inherits its singularities from 
\begin{equation}\label{T_simplified}
T(2^r x,1-2^{-r}) = T(2^r x e^{-x}) - x.
\end{equation}
Note that $x e^{-x}$ is real if and only if either $x$ is real or $x$ is of the form $y(\cot y + i)$ for some real $y$. Since $|y(\cot y +i)| = \big|\frac{y}{\sin y} \big| \geq 1$, the only part of the open unit disk that is mapped to the real axis by the function $x \mapsto 2^r x e^{-x}$ is the real interval $(-1,1)$. Moreover, this function is increasing on $(-1,1)$. It follows that the function $T(2^r x e^{-x})$ is analytic on the unit disk, except for a branch cut along the positive real axis starting at the solution $\eta_r = T(\frac{1}{2^r e})$ of the equation $2^r x e^{-x} = \frac1{e}$.

Since this is decreasing as a function of $r$, the term $r=k$ in~\eqref{eq:upper_est} dominates the rest. Now combine~\eqref{eq:tzu_explicit} and~\eqref{T_simplified} to obtain
$$t(2^k x, 1- 2^{-k}) = (1-x) T(2^k x, 1- 2^{-k}) - \frac{T(2^k x, 1- 2^{-k})^2}{2} = T(2^k x e^{-x}) - \frac {T(2^k x e^{-x})^2}{2} - x + \frac{x^2}{2}.$$
Now we can make use of~\eqref{eq:unrooted_expansion}. Observe also that we have the Taylor expansion
$$2^k x e^{-x} = \frac1{e} \Big(1 + \frac{1-\eta_k}{\eta_k} (x - \eta_k) + O \big( (x-\eta_k)^2 \big) \Big)$$
around the point $\eta_k$. Putting everything together, we find that
$$t(2^k x, 1- 2^{-k}) = \frac{(1-\eta_k)^2}{2} + \gamma \Big(1-\frac{x}{\eta_k}\Big) + 
\frac{2\sqrt{2}}{3} (1-\eta_k)^{3/2} \Big(1-\frac{x}{\eta_k}\Big)^{3/2} + O \Big( \Big(1-\frac{x}{\eta_k}\Big)^2 \Big)$$
for a constant $\gamma$ (that can also be determined, but it is irrelevant for us). Finally, in view of~\eqref{eq:fzu_explicit},
$$f(2^k x, 1- 2^{-k}) = e^{(1-\eta_k)^2/2} \bigg( 1 + \gamma\Big(1-\frac{x}{\eta_k}\Big)  + \frac{2\sqrt{2}}{3} (1-\eta_k)^{3/2} \Big(1-\frac{x}{\eta_k}\Big)^{3/2} + O \Big( \Big(1-\frac{x}{\eta_k}\Big)^2 \Big) \bigg).$$
Recall that all terms with $r < k$ in~\eqref{eq:upper_est} are asymptotically irrelevant since their smallest singularities are all greater than $\eta_k$. The factor $2^{\binom{k}{2}} \frac{x^k}{k!}$ in~\eqref{eq:upper_est} results in an additional factor $2^{\binom{k}{2}} \frac{\eta_k^k}{k!}$ in the dominant singular term of order $(1- x/\eta_k)^{3/2}$. Hence we obtain that \begin{equation}\label{eq:asymp-exp}
U_k(x) = a_k + b_k\Big(1-\frac{x}{\eta_k}\Big) + c_k \Big(1-\frac{x}{\eta_k}\Big)^{3/2} + O \Big( \Big(1-\frac{x}{\eta_k}\Big)^2 \Big)
\end{equation}
for suitable constants $a_k,b_k,c_k$. Specifically, $c_k = 2^{\binom{k}{2}}\frac{2\sqrt{2}}{3k!} e^{(1-\eta_k)^2/2}(1-\eta_k)^{3/2}\eta_k^k$. This is valid in the intersection of a neighbourhood of $\eta_k$ with the slit plane that has all real numbers greater than $\eta_k$ removed. Apart from this branch cut, $U_k(x)$ is an analytic function for $|x| < \eta_{k-1}$. Hence \cite[Theorem VI.4]{Flajolet2009analytic} yields
\begin{equation}\label{eq:asymp-formula}
[x^n]U_k(x) \sim \frac{c_k}{\Gamma(-3/2)} n^{-5/2} \eta_k^{-n}
\end{equation}
as $n \to \infty$, which implies the desired result.
\end{proof}

Next, we provide a lower bound for $\A_k$ that is of the same order as the upper bound of Lemma~\ref{lem:upper}.
\begin{lemma} \label{lem:lower}
Let $\eta_k = T(\frac{1}{2^k e})$ be as defined in the previous lemma. There exists a positive constant $K_2$ (depending on $k$) such that $\A_k$ contains at least $K_2 n^{-5/2} \eta_k^{-n}n!$ elements with $n$ vertices for all sufficiently large positive integers $n$.
\end{lemma}

\begin{proof}
We provide a matching lower bound in a similar way as in the previous lemma. Consider the set $\cS_k$ of all graphs consisting of a tree $T$ and a complete graph $K_k$ and a number of additional edges, each with one end in $T$ and the other in $K_k$, such that each leaf of $T$ is adjacent to at least one of the vertices in the complete graph. We show that almost all of these graphs belong to $\A_k$, i.e., they are $2$-connected. Later, we derive a lower bound for the number of graphs in $\cS_k$, which in turn yields a lower bound for the number of elements of $\A_k$.

Consider an element $G$ of $\cS_k$: we show that it is $2$-connected unless an exceptional situation occurs.
Indeed, if one of the tree vertices is removed, the tree decomposes into several connected components, each of which contains at least one leaf. Since each of the leaves needs to be adjacent to at least one of the vertices of the complete graph, the resulting graph is still connected. On the other hand, if one of the vertices of the complete graph $K_k$ is removed, the remaining graph consists of a tree and a complete graph $K_{k-1}$ (both of which are connected graphs), and these two graphs are still connected by at least one edge unless all edges connecting $T$ to the complete graph lead to the same vertex. Since this is the only scenario for which $G$ is not $2$-connected, we can expect most elements of $\cS_k$ to be $2$-connected. We will prove this below by obtaining the exponential generating function for the non-$2$-connected elements of $\cS_k$.

First, we can use the same reasoning that gave us~\eqref{eq:upper_est} to find that the exponential generating function for the set of graphs $\cS_k$ is
$$L_k(x) = \frac{x^k}{k!} \sum_{m \geq 1} \sum_{\substack{T \in \T \\ |T| = m}} \frac{(2^r x)^m}{m!} \Big( \frac{2^r-1}{2^r} \Big)^{\ell(T)} = \frac{x^k}{k!} t \big( 2^r x, 1-2^{-r}\big),$$
which has the same dominant singularity $\eta_k$ and an asymptotic expansion of the same form as~\eqref{eq:asymp-exp}, albeit with other coefficients. Hence the coefficients of $L_k$ satisfy an asymptotic formula of the form~\eqref{eq:asymp-formula} (with a different multiplicative constant).

However, some more care is needed to complete the proof: firstly, we need to subtract those graphs that are not $2$-connected because one of the vertices of the complete graph is an endpoint of all connecting edges. There must be an edge between this vertex and all leaves, and there might be further edges between this vertex and other tree vertices, but no other edges connecting the tree and the complete graph. Hence the exponential generating function for such graphs is
$$k \cdot \frac{x^k}{k!} \sum_{m \geq 1} \sum_{\substack{T \in \T \\ |T| = m}} \frac{x^m}{m!} 2^{m-\ell(T)} =  \frac{x^k}{(k-1)!} t \big( 2x, \tfrac12\big),$$
which has a greater radius of convergence (namely $\eta_1$) than $L_k$, hence its coefficients are negligibly small.

The second issue we need to take into consideration is the potential double-counting: for a given graph in $\A_k$, the $k$ vertices forming the complete graph may not be unique. We will show, however, that this only happens for a very small proportion of graphs in $\A_k$.

To this end, we consider the degrees of the 
vertices. Among all the possible combinations consisting of a tree, a complete graph $K_k$ and edges between the two as described above, pick one at random. The distribution of the degree of a vertex of the complete graph is almost a binomial distribution, meaning that the degree is concentrated around $\frac{n}{2}$. Consider e.g. the probability that such a vertex has degree at most $\frac{n}{3}$. The fact that the vertex is connected by an edge to all other vertices of the complete graph by default, and the fact that the probability to be connected to a leaf by an edge is slightly above $\frac12$, only decrease this probability compared to the binomial distribution. For the binomial distribution, we find the probability to be at most $e^{-n/36}$ using the Chernoff bound ${\mathbb P}[X\leq (1- \epsilon) (n/2)] \leq e^{-\epsilon^2 n/4}$ for $\epsilon=1/3$, which is exponentially small in $n$.
So the $k$ vertices that form the complete graph have a degree of at least $\frac{n}{3}$ for all but a negligible set of combinations.

On the other hand, let us estimate the number of combinations for which a tree vertex has large degree, say at least $\frac{n}{4}$. A simple upper bound will suffice: there are $\binom{n}{k}$ ways to distribute the labels, $n-k$ choices for the vertex with large degree, and at most $2^{k(n-k)}$ possibilities for the edges between tree and complete graph. The number of labelled trees with $n-k$ vertices for which a fixed vertex has degree $d$ is $\binom{n-k-2}{d-1}(n-k-1)^{n-k-d-1}$ (by \cite[(1.7.5)]{Harary1973graphical}). Hence we have the upper bound
\begin{align*}
(n-k) \binom{n}{k} 2^{k(n-k)} \sum_{d \geq n/4} \binom{n-k-2}{d-1} (n-k-1)^{n-k-d-1} &\leq \frac{n^{k+1}}{k!} 2^{k(n-k)} \cdot 2^{n-k-2} \cdot n^{n-k-n/4-1} \\
&= O (n^{3n/4} 2^{(k+1)n}),
\end{align*}
which by Stirling's formula is also negligibly small compared to $K_2n^{-5/2} \eta_k^{-n} n!$. It follows that for all but a negligible set of combinations, the vertices of the complete graph are the only vertices whose degree is at least $\frac{n}{3}$, which means that they are unique. This completes the proof of the lower bound.
\end{proof}

\begin{corollary}\label{cor:Bk}
For $k \geq 4$, the exponential generating function $B_k(z)$ associated with the set $\B_k$ of all possible blocks of graphs in $\G_k$ has radius of convergence $\eta_k$. Moreover,
$$\lim_{x \to \eta_k^-} B''_k(x) = \infty.$$
\end{corollary}

\begin{proof}
We make use of the results of Gim\'enez and Noy \cite{Gimenez2009asymptotic} on the enumeration of planar graphs. Specifically, the number of $2$-connected labelled planar graphs with $n$ vertices is asymptotically given by  \cite[Theorem 9.13]{Drmota2009random}
$$p_n^{(2)} \sim \alpha n^{-7/2} \beta^{-n}n!$$
for $\alpha \approx 0.37042 \cdot 10^{-5}$ and $\beta \approx 0.03819$. Since $\beta > \eta_4 \approx 0.02354 \geq \eta_k$ for $k \geq 4$, it follows that the planar blocks form a negligible portion of $\B_k$, as the number of elements in $\A_k$ grows exponentially faster. The statement on the radius of convergence and the behaviour of the second derivative now follow immediately from the previous two lemmas.
\end{proof}

Now we are ready to complete the proof of the main theorem.

\begin{proof}[Proof of Theorem~\ref{thm:main}]

It is well known (see \cite[Section VI.7]{Flajolet2009analytic} or \cite[Section 3.1.4]{Drmota2009random}) that a function $y(x)$ that satisfies a functional equation of the form $y(x) = x \Phi(y(x))$, where $\Phi$ has positive coefficients, has a dominant square root singularity at $\rho = \tau/\Phi(\tau)$, where $\tau$ is the positive real solution to the equation $t\Phi'(t) = \Phi(t)$, provided such a solution exists inside the circle of convergence of $\Phi$.

In our situation, where we consider the exponential generating function $\C_k^{\bullet}$ of rooted connected graphs in $\G_k$, we can let $\Phi(t) = \exp(B_k'(t))$ in view of~\eqref{eq:subcritical_funeq}, so the equation $t \Phi'(t) = \Phi(t)$ reduces to $t B_k''(t) = 1$. In view of Corollary~\ref{cor:Bk}, we have $\lim_{t \to \eta_k^-} t B_k''(t) = \infty$. Since we also have $\lim_{t \to 0^+} t B_k''(t) = 0$ and $t B_k''(t)$ is continuous and increasing as a function of $t$, it follows from the intermediate value theorem that there is indeed a unique value $\tau_k \in (0, \eta_k)$ such that $\tau_k B_k''(\tau_k) = 1$. 

Consequently, $C^{\bullet}_k$  has its dominant square root singularity at $\rho_k = \tau_k \exp(-B_k'(\tau_k))$, and $C^{\bullet}_k(\rho_k) = \tau_k < \eta_k$. This proves that $\G_k$ is a subcritical family.
\end{proof}

\section{The probability of 2-connectedness} \label{sec 2 conn}

In the previous section we obtained asymptotics for the number of 2-connected $k$-apex forests with $n$ vertices. The corresponding asymptotics for the number $\Z_{k,n}$  of all $n$-vertex $k$-apex forests, not necessarily 2-connected, was determined by Kurauskas and McDiarmid \cite{KurMcDRan}:

\begin{theorem}[\cite{KurMcDRan}]\label{KMtheorem}
	$|\Z_{k,n}| \approx c_k n^{-5/2} \zeta_k^{n}n!$, where $\zeta_k=e2^k$ and $c_k = \left( 2^{k+1 \choose 2}e^k k! \right)^{-1}$.
\end{theorem}

(Unlike our Lemmas~\ref{lem:upper} and~\ref{lem:lower} that leave the constants $K_1$ and $K_2$  unknown, this result provides an exact constant $c_k$.)

Comparing this with Lemmas~\ref{lem:upper} and~\ref{lem:lower} immediately yields the asymptotics of the probability for the uniform random $n$-vertex $k$-apex forest to be 2-connected:

\begin{corollary}
The uniformly random  $k$-apex forest with $n$ vertices is 2-connected with probability $P_n$ satisfying
$$P_n = \Theta(( \zeta_k \eta_k)^{-n}) = \Theta(( e2^k \eta_k)^{-n}) = \Theta(e^{-\eta_kn}),$$
where $\eta_k$ is defined as the smallest positive solution to the equation $2^k \eta = e^{\eta-1}$.
\end{corollary} \qed

We remark that it is not straightforward to deduce Lemmas~\ref{lem:upper} and~\ref{lem:lower} from Theorem~\ref{KMtheorem} using  the typical number of leaves of a uniform random forest and the fact that a $k$-apex forest can only be 2-connected if each leaf of its underlying forest is connected to at least one of the apex vertices: calculations show that the rate of decay of $P_n$ is not determined by the typical number of leaves of the uniform random forest with $n$ vertices (for the uniform random tree the number of leaves divided by $n$ converges in probability to $1/e$ \cite{GolSho}, and this remains true for the uniform random forest), and is hence influenced by the `unlikely' forests with much fewer leaves. 

\section{Extension to all minor-closed families of graphs} \label{secdisc}

Looking back over the above proof, we observe that very little information about the family of planar graphs was actually used: we only used the fact that it is minor-closed (in the proof of Proposition~\ref{prop_add_mc}), and that its exponential generating function has a non-zero radius of convergence (in the proof of Corollary~\ref{cor:Bk}). By a theorem of Norine, Seymour, Thomas and Wollan \cite{NSTW}, every proper minor-closed family of graphs has the property that the number $g_n$ of labelled graphs in the family satisfies the inequality
$$g_n \leq n! \cdot c^n$$
for some positive constant $c$. Hence we can immediately extend Theorem~\ref{thm:main} as follows.

\begin{theorem}\label{thm:general}
Every proper minor-closed family $\M$ of graphs is contained in a minor-closed, addable, subcritical family of graphs.
\end{theorem}

\begin{proof}
Analogous to the proof of Theorem~\ref{thm:main}: consider the family of graphs whose blocks are either single edges, $2$-connected graphs in $\M$, or $2$-connected $k$-apex forests, and take $k$ large enough so that $\eta_k < 1/c$ (which is possible since $\eta_k \to 0$ as $k \to \infty$). The resulting family will be minor-closed, addable, and subcritical.
\end{proof}

In view of Theorem~\ref{thm:general}, there is little hope to achieve a full characterisation of minor-closed subcritical graph classes in terms of their forbidden minors. However, as mentioned earlier, it is still plausible that the other direction of Noy's conjecture holds, i.e. that every addable, minor-closed graph class with a forbidden planar minor is subcritical.

\bibliographystyle{abbrv}
\bibliography{Subcritical}

\begin{thebibliography}{10}

\bibitem{BePaStDeg}
N.~Bernasconi, K.~Panagiotou, and A.~Steger.
\newblock The {Degree} {Sequence} of {Random} {Graphs} from {Subcritical}
  {Classes}.
\newblock {\em Combin. Probab. Comput.}, 18(5):647--681, 2009.

\bibitem{CGHJK}
R.~M. Corless, G.~H. Gonnet, D.~E.~G. Hare, D.~J. Jeffrey, and D.~E. Knuth.
\newblock On the {L}ambert {$W$} function.
\newblock {\em Adv. Comput. Math.}, 5(4):329--359, 1996.

\bibitem{Drmota2009random}
M.~Drmota.
\newblock {\em Random trees}.
\newblock Springer, Vienna, 2009.

\bibitem{DFKKR}
M.~Drmota, E.~Fusy, M.~Kang, V.~Kraus, and J.~Rué.
\newblock Asymptotic study of subcritical graph classes.
\newblock {\em {SIAM} J. Discret. Math.}, 25(4):1615--1651, 2011.

\bibitem{drmota13}
M.~Drmota and M.~Noy.
\newblock Extremal parameters in sub-critical graph classes.
\newblock In {\em A{NALCO}13---{M}eeting on {A}nalytic {A}lgorithmics and
  {C}ombinatorics}, pages 1--7. SIAM, Philadelphia, PA, 2013.

\bibitem{DrmRamRue}
M.~Drmota, L.~Ramos, and J.~Ru\'e.
\newblock Subgraph statistics in subcritical graph classes.
\newblock {\em Random Structures Algorithms, to appear}.
\newblock arXiv: 1512.08889.

\bibitem{Flajolet2009analytic}
P.~Flajolet and R.~Sedgewick.
\newblock {\em Analytic combinatorics}.
\newblock Cambridge University Press, Cambridge, 2009.

\bibitem{Georgakopoulos2015limits}
A.~Georgakopoulos and S.~Wagner.
\newblock Limits of subcritical random graphs and random graphs with excluded
  minors.
\newblock 2015.
\newblock arXiv:1512.03572.

\bibitem{Gimenez2009asymptotic}
O.~Gim{\'e}nez and M.~Noy.
\newblock Asymptotic enumeration and limit laws of planar graphs.
\newblock {\em J. Amer. Math. Soc.}, 22(2):309--329, 2009.

\bibitem{GolSho}
C.~Goldschmidt.
\newblock A short introduction to random trees.
\newblock {\em Mongolian Mathematical Journal}, 20(1):53--72, 2016.

\bibitem{Harary1973graphical}
F.~Harary and E.~M. Palmer.
\newblock {\em Graphical enumeration}.
\newblock Academic Press, New York-London, 1973.

\bibitem{JosDer}
M.~Josuat-Verg{\`e}s.
\newblock Derivatives of the tree function.
\newblock {\em The Ramanujan Journal}, 38(1):1--15, Oct 2015.

\bibitem{KurMcDRan}
V.~Kurauskas and C.~McDiarmid.
\newblock Random graphs with few disjoint cycles.
\newblock {\em Combinatorics, Probability and Computing}, 20(5):763–775,
  2011.

\bibitem{NSTW}
S.~Norine, P.~Seymour, R.~Thomas, and P.~Wollan.
\newblock Proper minor-closed families are small.
\newblock {\em J. Combin. Theory Ser. B}, 96(5):754--757, 2006.

\bibitem{Noy2014random}
M.~Noy.
\newblock Random planar graphs and beyond.
\newblock In {\em Proceedings of the International Congress of Mathematicians,
  Seoul 2014, Volume IV}, pages 407--430, 2014.

\bibitem{Panagiotou2016scaling}
K.~Panagiotou, B.~Stufler, and K.~Weller.
\newblock Scaling limits of random graphs from subcritical classes.
\newblock {\em Ann. Probab.}, 44(5):3291--3334, 2016.

\bibitem{GMV}
N.~Robertson and P.~Seymour.
\newblock Graph minors. {V}. {E}xcluding a planar graph.
\newblock {\em J.~Combin.\ Theory (Series B)}, 41:92--114, 1986.

\bibitem{StuRan}
B.~Stufler.
\newblock Random enriched trees with applications to random graphs.
\newblock 2015.
\newblock arXiv: 1504.02006.

\end{thebibliography}

\end{document}